\documentclass[11pt]{article}
\usepackage{latexsym,amsmath,amsthm,verbatim,ifthen,amssymb}
\usepackage[latin1]{inputenc}
\usepackage[all]{xy}
\usepackage{graphicx}
\usepackage{epsfig}

\newdir{ >}{{}*!/-12pt/@{>}}
\newdir{ (}{!/-3pt/@_{(} } 

\newdir{ )}{!/-3pt/@^{(} } 

\newtheorem{theorem}{Theorem}[section]

\newtheorem{corollary}[theorem]{Corollary}
 \newtheorem{lemma}[theorem]{Lemma}
 \newtheorem{proposition}[theorem]{Proposition}
 \theoremstyle{definition}
 \newtheorem{definition}[theorem]{Definition}
 \theoremstyle{remark}
 \newtheorem{remark}[theorem]{Remark}

 \numberwithin{equation}{subsection}

\newcommand{\cat}{{\sf {cat}}} 
\newcommand{\secat}{{\sf {secat}}} 
\newcommand{\tc}{{\sf {TC}}} 

\newcommand{\D}{{\sf {D}}}

\begin{document}

\title{Relative sectional category revisited}

\author{J.M. Garc\'{\i}a-Calcines\footnote{Universidad de La Laguna,
Facultad de Ciencias, Departamento de Matem\'aticas,
Estad\'{\i}stica e I.O., 38271 La Laguna, Spain. E-mail:
\texttt{jmgarcal@ull.edu.es}} }

\maketitle

\begin{abstract}
The concept of relative sectional category expands upon classical sectional category theory by incorporating the pullback of a fibration along a map. Our paper aims not only to explore this extension but also to thoroughly investigate its properties. We seek to uncover how the relative sectional category unifies several homotopic numerical invariants found in recent literature. These include the topological complexity of maps according to Murillo-Wu or Scott, relative topological complexity as defined by Farber, and homotopic distance for continuous maps in the sense of Mac\'{\i}as-Virg\'os and Mosquera-Lois, among others.
\end{abstract}

\vspace{0.5cm}
\vspace{0.2cm}

\section*{Introduction}
For a fibration $p:E\rightarrow B,$ the sectional category, denoted $\secat (p)$, is defined as the smallest integer $k$ such that $B$ admits a cover composed of $k+1$ open subsets, each of which has a section of $p$. By requiring a homotopy section instead of just a section, this definition extends the concept of sectional category to any map. It serves as both a variant and a generalization of the Lusternik-Schnirelmann category (or L-S category) since $\secat (p)=\cat(B)$ when $E$ is contractible. Introduced for fibrations by Schwarz \cite{Sch} in the 1960s under the name ``genus" (later renamed by James), the sectional category has emerged as a valuable tool not only in questions concerning bundle classification and the embedding problem or the complexity of root-finding problems for algebraic equations but also, more recently, in the study of motion planning problems in robotics.
To explore Hopf invariants within the framework of sectional category theory, J. Gonz\'alez, M. Grant, and L. Vandembroucq introduced in \cite{GGV} a notion of relative sectional category . For a given a fibration $p:E\rightarrow B$ and a continuous map $f:X\rightarrow B,$ the relative sectional category of $p$ with respect to $f$ is defined as the sectional category of $f^*p$, where $f^*p$ denotes the pullback of $p$ along $p$. Although this invariant exhibits intriguing properties, its exploration remains somewhat limited in their work, where the authors focus on the case where $f$ is an inclusion. This paper seeks to explore further the study of relative sectional category, showcasing its potential to unify various homotopy numerical invariants. These include the classical sectional category, the Lusternik-Schnirelmann category of maps, the relative topological complexity of M. Farber \cite{F2}, the topological complexity of maps as defined by Murillo-Wu and Scott (refer to \cite{M-W}, \cite{Sct}), and the homotopic distance as defined by Mac\'{\i}as-Virg\'os and Mosquera-Lois \cite{MV-ML}, among others.

We have structured the content of this article as follows:
Firstly, in Section 1, we introduce the main concept of our study. Our approach is broader than the original framework presented by J. Gonz\'alez, M. Grant, and L. Vandembroucq, specifically employing the homotopical pullback construction to consider maps that are not necessarily fibrations.
In Section 2, we provide a broad and representative sampling of homotopical numerical invariants, utilizing the notion of relative sectional category.
In Section 3, we establish crucial properties concerning the relative sectional category. We begin by recalling significant bounds provided by J. Gonz\'alez, M. Grant, and L. Vandembroucq, and then introduce new properties, including homotopical invariance and an inequality involving products. Additionally, we also present several intriguing inequalities that hold under certain conditions.
In Section 4, we offer an axiomatized characterization reminiscent of the Whitehead-Ganea framework, where the join construction, a blend of pushout and homotopical pullback of two maps, plays a pivotal role.
Finally, in Section 5, we explore a generalized version of the relative sectional category, considering coverings that are not necessarily open. We demonstrate that equality between the relative sectional category and its generalized version holds provided that all involved spaces have the homotopy type of a CW-complex, or equivalently, an ANR space.

\section{Introducing relative sectional category}

Homotopy limits and colimits are indispensable tools in homotopy theory.
In this work, we will only use homotopy pullbacks and pushouts, and presume the reader's familiarity with these objects along with their key properties. Nonetheless, we will offer a brief recap of these concepts here, and for those interested in further exploration, we provide as references \cite{M}, \cite{D}, and \cite{B-K}.
Given a homotopy commutative square
$$\xymatrix{
  P \ar[d]_{\alpha } \ar[r]^{\beta } & B \ar[d]^{g} \\
  A \ar[r]_{f} & C   } $$
\noindent equipped with a homotopy $H:f\circ \alpha \simeq g\circ \beta $,
another homotopy commutative square exists
$$\xymatrix{
  E_{f,g} \ar[d]_{p} \ar[r]^{q}  & B \ar[d]^{g} \\
A \ar[r]_{f} & C   }$$
\noindent where $E_{f,g}$ is the space defined as $E_{f,g}=\{(a,b,\theta )
\in A\times B\times C^I \, ; \: \theta (0)=f(a)\, , \:
\theta (1)=g(b) \}$ and $p,q$ are the obvious restrictions of the
projections. This square is equipped with the homotopy $G$ given
by $(a,b,\theta,t) \mapsto \theta(t).$ Then, a map (referred to as the \textit{whisker map})
$w:P\rightarrow E_{f,g}$ exists, defined by
$w(x)=(\alpha (x), H(x,-), \beta (x))$ satisfying $p\circ w=\alpha
,$ $q\circ w=\beta $ and $G\circ (w\times id)=H.$ The first square is termed a
\textit{homotopy pullback} if the whisker map $w$ is a
homotopy equivalence; the second square is referred to as the \textit{standard homotopy pullback}
of $f$ and $g.$ The Eckmann-Hilton dual notion corresponds to the \textit{homotopy
pushout}. Homotopy pullbacks and homotopy pushouts may also be equivalently
defined through factorization properties or by the weak universal
property of homotopy (co-)limits.

\medskip

We will also assume that the reader is familiar with sectional category theory. For convenience, we will summarize some well-known key properties of the sectional category that will be used throughout this work in the following result:

\begin{lemma}
Consider the following homotopy commutative square of spaces and maps:
$$\xymatrix{
{E} \ar[r]^{\alpha } \ar[d]_p & {E'} \ar[d]^{p'} \\
{B} \ar[r]_{\beta } & {B'.} }$$
\begin{enumerate}
\item[(a)] If $\alpha $ and $\beta $ are homotopy equivalences, then $\mbox{secat}(p)=\mbox{secat}(p')$.

\item[(b)] If $B=B'$ and $\beta =id_B$, then $\mbox{secat}(p)\geq \mbox{secat}(p')$.

\item[(c)] If the square is a homotopy pullback, then $\mbox{secat}(p)\leq \mbox{secat}(p')$.
\end{enumerate}
\end{lemma}

\bigskip
And now we proceed to introduce the concept of relative sectional category, first defined by J. Gonz\'alez, M. Grant, and L. Vandembroucq in \cite{GGV}. Here we will provide a slight variation in order to generalize to maps that are not necessarily fibrations.

\begin{definition}
Consider a cospan $K\stackrel{\varphi }{\longrightarrow }X\stackrel{p}{\longleftarrow }A$ of maps. A subset $Z\subset K$ is said to be $\varphi $-sectional if there exists a map $s:Z\rightarrow A$ such that $p\circ s\simeq \varphi _{|Z}$; that is, there is a homotopy commutative diagram
$$\xymatrix{
{Z} \ar[rr]^{\varphi _{|Z}} \ar[dr]_s & & {X} \\
 & A. \ar[ur]_p & }$$
The \emph{sectional category of $p$ relative to $\varphi $}, $\secat _{\varphi }(p)$, is the least non-negative integer $n$ such that there is an open cover $K=U_0\cup U_1\cup \cdots \cup U_n$ by $n+1$ open $\varphi $-sectional subsets. If such an integer does not exist then we set $\secat _{\varphi }(p)=\infty .$
\end{definition}

It is obvious that given $p,q:A\rightarrow X$ and $\varphi ,\psi :K\rightarrow X$ maps such that $p\simeq q$ or $\varphi \simeq \psi $, then $\secat _{\varphi }(p)=\secat _{\psi }(q).$

\begin{proposition}\label{pullback}
Let $K\stackrel{\varphi }{\longrightarrow }X\stackrel{p}{\longleftarrow }A$ be a cospan of maps. Then $\secat _{\varphi }(p)=\secat (\overline{p})$, where $\overline{p}$ denotes the base change of the map $p$ in the homotopy pullback of $p$ along $\varphi $:
$$\xymatrix{
{P_{\varphi ,p}} \ar[r]^{\overline{\varphi }} \ar[d]_{\overline{p}}  & {A} \ar[d]^p \\
{K} \ar[r]_{\varphi }  & {X.}  }$$
\end{proposition}

\begin{proof}
Let $U\subset $ be an open subset together with a map $s:U\rightarrow A$ satisfying $p\circ s\simeq \varphi _{|U}.$ Then, by the weak universal property of the homotopy pullback we have an induced map
$$
\xymatrix{
{U}  \ar@/^1pc/[drr]^{s} \ar@/_1pc/[ddr]_{inc} \ar@{.>}[dr]^{\sigma } & & \\
 & {P_{\varphi ,p}} \ar[r] \ar[d]_{\overline{p}} & {A} \ar[d]^p \\
 & {K} \ar[r]_{\varphi } & {X}
}$$
\noindent which is a local homotopy section of the map $\overline{p}$ on the open subset $U$. Conversely, given such a local homotopy section $\sigma :U\rightarrow P_{p,\varphi }$ the composite of $\sigma $ with $P_{\varphi ,p} \rightarrow A$ (the base change of $\varphi $) gives us that $U$ is a $\varphi $-sectional open subset.
\end{proof}

\begin{remark}
We want to point out that, in the previous result, we did not explicitly specify the homotopy pullback space $P_{\varphi, p}$ nor the map $\overline{p}$ over the space $K$, as different choices yield homotopically equivalent continuous maps. In this sense, one could consider, if desired, the standard homotopy pullback of $\varphi$ with $p$, $E_{\varphi ,p}$, although this is not strictly necessary for the development of our theory.
\end{remark}

Observe that, in particular, when $p$ is a fibration we can take the honest pullback,
$$\xymatrix{
{K\times _X A} \ar[r] \ar[d]_{\varphi ^*p}  & {A} \ar[d]^p \\
{K} \ar[r]_{\varphi }  & {X}  }$$
\noindent which is also a homotopy pullback. In this case $\varphi ^*p$ is a fibration as well.

\begin{corollary}
Let $K\stackrel{\varphi }{\longrightarrow }X\stackrel{p}{\longleftarrow }A$ be a cospan of maps, where $p$ is a fibration. Then
$$\secat _{\varphi }(p)=\secat (\varphi ^*p).$$
\end{corollary}

\begin{remark}
In practice, when dealing with a setup where $K\stackrel{\varphi }{\longrightarrow }X\stackrel{p}{\longleftarrow }A$ forms a cospan of maps where $p$ is not a fibration, to compute $\secat _{\varphi }(p)$, one first replaces $p$ with its fibrational substitute, $p':A'\rightarrow X$. Then, the honest pullback of $p'$ along $\varphi $ is considered:
$$\xymatrix{
{P_{\varphi ,p'}} \ar[r] \ar[d]_{\varphi ^*p'} & {A'} \ar[d]^{p'} & {A} \ar[l]_{\simeq } \ar[dl]^p \\
{K} \ar[r]_{\varphi } & {X} &
}$$
\noindent so that we can consider $\secat _{\varphi }(p)=\secat (\varphi ^*p')$.
\end{remark}

\section{Remarkable examples of relative sectional category}

The notion of relative sectional category encompasses a large number of interesting examples. Let us take a few of them.
We first note that, given a map $p:A\rightarrow X$ it is obvious that $$\secat (p)=\secat _{id _X}(p).$$ In particular, both topological complexity and Lusternik-Schnirelmann category are particular cases of relative sectional category. Moreover, given $\varphi :A\rightarrow X$ be a map we obviously have the identity $$\cat (\varphi )=\secat _{\varphi }(*\rightarrow X)$$ \noindent where $\cat (\varphi )$ denotes the usual Lusternik-Schnirelmann category of the map $\varphi $.

\medskip
Apart from the aforementioned direct examples, perhaps the foremost highlighted example of the relative sectional category is the relative topological complexity in the sense of Farber \cite{F2}, also known as subspace topological complexity. Given $X$ a (path-connected) space and $A\subset X\times X$ a subspace, the subspace topological complexity of $A$ in $X$, denoted by $\tc _X(A)$, is defined as
$$\tc _X(A):=\secat ((\pi _X)_{|A})$$
\noindent where $(\pi _X)_{|A}:\pi _X^{-1}(A)\rightarrow A$ denotes the restriction to paths in $X$ whose pair of initial and terminal points lies in $A$. Observe that $\pi _X^{-1}(A)$ fits in the (homotopy) pullback
$$\xymatrix{
{\pi _X^{-1}(A)}  \ar[d]_{(\pi _X)_{|A}} \ar@{^{(}->}[r] & {X^I} \ar[d]^{\pi _X} \\
{A} \ar@{^{(}->}[r]_{inc_A} & {X\times X}
}$$
\noindent and therefore, it is the relative sectional category $\tc _X(A)=\secat _{inc_A}(\pi _X).$ The concept was originally introduced by M. Farber \cite{F2}, and has since found application in the works of various authors. Notable among them are M. Grant \cite{Gr} and Boehnke-Scheirer-Xue \cite{B-S-X}. An intriguing variation is the notion of the topological complexity of a pair, as introduced and developed by R. Short in \cite{Sh}; for a given space $X$ and a subspace $Y \subset X$, R. Short defined $\tc(X,Y)$ as $\tc_X(X \times Y)$. A natural extension of this concept is the relative sectional category with respect to an inclusion, which has been utilized by J. Gonz\'alez, M. Grant, and L. Vandembroucq in \cite{GGV} for their investigation of Hopf invariants for sectional category. Specifically, for a fibration $p: A \rightarrow X$ and a subspace $K \subset X$, they consider $\secat_K(p)$ as the relative sectional category $\secat_{\text{inc}_K}(p)$, where $\text{inc}_K: K \hookrightarrow X$ denotes the inclusion.

\bigskip
Another source of examples of relative sectional category is given by the study of topological complexity for maps. In this regard, the topological complexity of a map $f:X\rightarrow Y$ in the sense of Scott \cite{Sct}, $\tc ^{Sc}(f)$, also known as the pullback topological complexity of $f$, is introduced as the sectional category of the base change $\pi _Y^*:P_{f\times f,\pi _Y}\rightarrow X\times X$ in the pullback
$$\xymatrix{
{P_{f\times f,\pi _Y}} \ar[d]_{\pi _Y^*} \ar[rr] & & {Y^I} \ar[d]^{\pi _Y} \\
{X\times X} \ar[rr]_{f\times f} & & {Y\times Y}
}$$
\noindent where $\pi _Y$ is the fibration defined as $\pi _Y(\alpha )=(\alpha (0),\alpha (1)),$ for all $\alpha \in Y^I.$ Therefore,  $\tc ^{Sc}(f)=\secat _{f\times f}(\pi _Y)$. Observe that we also have the identity $\tc ^{Sc}(f)=\secat _{f\times f}(\Delta _Y)$, where $\Delta _Y:Y\rightarrow Y\times Y$ denotes the diagonal map.
Similarly, Scott introduced the mixed topological complexity of a map $f:X\rightarrow Y$, denoted as $\tc ^{\frac{1}{2}}(f)$, as the sectional category of the base change in the pullback
$$\xymatrix{
{P_{id_Y\times f,\pi _Y}} \ar[d]_{\pi _Y^*} \ar[rr] & & {Y^I} \ar[d]^{\pi _Y} \\
{Y\times X} \ar[rr]_{id_Y\times f} & & {Y\times Y.}
}$$ Therefore, $\tc ^{\frac{1}{2}}(f)=\secat _{id_Y\times f}(\pi _Y)=\secat _{id_Y\times f}(\Delta _Y)$.

\bigskip
A. Murillo and J. Wu defined in \cite{M-W} an engaging notion called $f$-sectional category. Namely, given maps $E\stackrel{p}{\longrightarrow }B\stackrel{f}{\longrightarrow }X$
an open subset $U\subset X$ is said to be $f$-categorical if there is a map $s:U\rightarrow E$ such that $fps\simeq f_{|U}$, that it, the following triangle is commutative up to homotopy:
$$\xymatrix{
{U} \ar[rr]^{f_{|U}} \ar[dr]_s & & {X} \\
 & {E.} \ar[ur]_{fp} }$$

The $f$-sectional category of $p$, denoted here as $\secat _f^{MW}(p),$ is the least integer $n$ for which $B$ admits a covering of $n+1$ $f$-categorical open sets. If no such integer exists then we set $\secat _f^{MW}(p)=\infty .$ It is immediate that we have a cospan $B\stackrel{f}{\longrightarrow }X\stackrel{fp}{\longleftarrow }E$ and that the $f$-sectional category is nothing else than the relative sectional category
$\secat _f^{MW}(p)=\secat _f(fp).$
This way, another way to characterize the $f$-sectional category is by considering the sectional category of the map $\overline{fp}$ within the homotopy pullback of $fp$ along $f$
$$\xymatrix{
{P_{f,fp}} \ar[r] \ar[d]_{\overline{fp}} & {E} \ar[d]^{fp} \\
{B} \ar[r]_f & {X.} }$$
The $f$-sectional category provided the foundational framework for introducing what is known as the topological complexity of the work map. In essence, for a given map $f:X\rightarrow Y$, the topological complexity of $f$ according to Murillo-Wu is defined as $$\tc ^{MW}(f):=\secat ^{MW} _{f\times f}(\pi _X)=\secat  _{f\times f}((f\times f)\pi _X)$$ \noindent where $\pi _X:X^I\rightarrow X\times X$ represents the bi-evaluation fibration. It's worth noting that Scott's topological complexity aligns numerically with the homotopy invariant version of Murillo-Wu's complexity, despite their different approaches to modeling motion planning problems.

\bigskip
Finally, an interesting example of relative sectional category is given by the so-called homotopic distance between two maps, a concept introduced by E. Mac\'{\i}as-Virg\'os and D. Mosquera-Lois in \cite{MV-ML}. Consider two maps, $\varphi, \psi :X\rightarrow Y.$ The homotopic distance $\D(\varphi,\psi)$ between $\varphi$ and $\psi $ is the smallest non-negative integer $n\geq 0$ for which there exists an open covering $X=U_0\cup \cdots \cup U_n$ such that the restrictions $\varphi |_{U_j}$ and $\psi |_{U_j}$ are homotopic maps, for all $j\in \{0,1,\cdots ,n\}.$ If no such covering exists, then we set $\D(\varphi,\psi) =\infty .$

There is a useful characterization in terms of the sectional category. To establish this, we consider the following pullback:
$$\xymatrix{
{\mathcal{P}(\varphi ,\psi)} \ar[d]_{\pi _Y^*} \ar[r] & {Y^I} \ar[d]^{\pi _Y} \\
{X} \ar[r]_{(\varphi ,\psi )} & {Y\times Y.} }$$
Then, E. Mac\'{\i}as-Virg\'os and D. Mosquera-Lois proved in \cite[Th. 2.7]{MV-ML}
$\D(\varphi,\psi)=\secat (\pi _Y^*:\mathcal{P}(\varphi ,\psi)\rightarrow X)$. Therefore, with our terminology, we have that
$$\D(\varphi,\psi)=\secat _{(\varphi ,\psi)}(\pi _Y)=\secat _{(\varphi ,\psi)}(\Delta _Y).$$

\section{Properties of relative sectional category}

In this subsection, we will explore interesting properties of the relative sectional category.
We begin with the following result, which has already been established in \cite{GGV}, providing elegant bounds for the relative sectional category. However, here we present their proofs for the reader's convenience.

\begin{proposition}\cite[Prop. 3.8]{GGV}
Let $K\stackrel{\varphi }{\longrightarrow }X\stackrel{p}{\longleftarrow }A$ be a cospan of maps. Then:
\begin{enumerate}
\item Let $x_1,\cdots ,x_k\in H^*(X)$ be cohomology classes with any coefficients such that $p^*(x_1)=\cdots =p^*(x_k)=0$ and $\varphi ^*(x_1...x_k)\neq 0.$ Then
$\secat _{\varphi }(p)\geq k.$

\item If $p$ is an $r$-equivalence ($r\geq 0$), $K$ and $X$ are path-connected spaces and $K$ has the homotopy type of a CW-complex, then
$$\secat _{\varphi }(p)\leq \frac{hdim(K)}{r+1}.$$
Here, $hdim(K)$ denotes the homotopy dimension of $K$, which refers to the smallest dimension of CW-complexes having the homotopy type of $K$.

\end{enumerate}
\end{proposition}

\begin{proof}
In order to prove the first assertion, let us consider $y_i = \varphi^*(x_i) \in H^*(K)$, for $i \in {1,...,k}$. Then, for each $i$:
$$\overline{p}^*(y_i)=\overline{p}^*(\varphi ^*(x_i))=\overline{\varphi }^*(p^*(x_i))=\overline{\varphi }^*(0)=0$$
\noindent which implies $y_i\in \mbox{Ker}(\overline{p}^*),$ for all $i\in \{1,...,k\}$. Moreover, the cup product $y_1...y_k=p^*(x_1)...p^*(x_k)=p^*(x_1...x_k)\neq 0.$ Thus,
from the general theory of sectional category \cite{Sch}, we have $$k\leq \mbox{nil}\hspace{2pt}\mbox{Ker}(\overline{p}^*)\leq \secat \hspace{2pt}({\overline{p}})=\secat _{\varphi }(p).$$

To establish the second assertion, observe that $\overline{p}$ is also an $r$-equivalence. Equivalently, the homotopy fiber of $\overline{p}$ is $(r-1)$-connected. Again, the result follows from the general theory of sectional category applied to $\overline{p}.$
\end{proof}

Moving forward we prove now that the relative sectional category is a homotopy invariant, in the following sense:

\begin{proposition}\label{h-invariant}
Consider the following homotopy commutative diagram
$$
\xymatrix{
{K} \ar[r]^{\varphi } \ar[d]^{\simeq }_{\alpha } & {X} \ar[d]^{\simeq }_{\beta}  & {A} \ar[d]^{\simeq }_{\gamma } \ar[l]_{p} \\
{L} \ar[r]_{\psi }  & {Y}  & {L}  \ar[l]^{q} }
$$
where $\alpha ,\beta $ and $\gamma $ are homotopy equivalences. Then $\secat _{\varphi }(p)=\secat _{\psi}(q).$
\end{proposition}

\begin{proof}
By virtue of the weak universal property of the homotopy pullback, a whisker map $\theta: E_{p,\varphi} \rightarrow E_{q,\psi}$ exists, rendering the following cube homotopy commutative, where the bottom and top faces are homotopy pullbacks:

$$\xymatrix@!0{
 {P_{\varphi ,p }} \ar@{.>}[dd]_{\theta } \ar[rr] \ar[dr] & & {A} \ar[dr] \ar[dd]|!{[d];[d]}\hole &  \\
  & {K} \ar[rr] \ar[dd] & & {X} \ar[dd]     \\
  {P_{\psi ,q}} \ar[rr]|!{[r];[r]}\hole \ar[dr] & & {B} \ar[dr] & \\
  & {L} \ar[rr] & &  {Y.}      }$$
Given that the vertical maps $\alpha$, $\beta$, and $\gamma$ are homotopy equivalences, it follows that $\theta$ must also be a homotopy equivalence. Hence, we achieve a homotopy commutative square where $\overline{p}$ and $\overline{q}$ are connected through homotopy equivalences:
$$\xymatrix{
{P_{\varphi ,p}} \ar[r]^{\theta }_{\simeq } \ar[d]_{\overline{p}} & {P_{\psi ,q }} \ar[d]^{\overline{q}} \\
{K} \ar[r]^{\simeq }_{\alpha } & {L.}
}$$
This way, we have $\secat _{\varphi }(p)=\secat(\overline{p})=\secat (\overline{q})=\secat _{\psi }(q).$
\end{proof}

An interesting inequality is also reflected in the following general situation:

\begin{proposition}\label{ineq}
Consider a homotopy commutative diagram of the form
$$\xymatrix{
 & {X} \ar[d]^{\beta } & {A} \ar[d]^{\gamma } \ar[l]_p \\
{K} \ar[ur]^{\varphi } \ar[r]_{\psi } & {Y} & {B.} \ar[l]^q
}$$
Then, $\secat _{\psi}(q)\leq \secat _{\varphi }(p).$
\end{proposition}

\begin{proof}
We actually have a homotopy commutative diagram
$$\xymatrix{
{K} \ar[r]^{\varphi } \ar@{=}[d] & {X} \ar[d]^{\beta } & {A} \ar[d]^{\gamma } \ar[l]_p \\
{K} \ar[r]_{\psi } & {Y} & {B} \ar[l]^q
}$$
\noindent which gives rise to the following commutative cube
$$\xymatrix@!0{
 {P_{\varphi ,p }} \ar@{.>}[dd]_{\theta } \ar[rr] \ar[dr] & & {A} \ar[dr] \ar[dd]|!{[d];[d]}\hole &  \\
  & {K} \ar[rr] \ar@{=}[dd] & & {X} \ar[dd]     \\
  {P_{\psi ,q}} \ar[rr]|!{[r];[r]}\hole \ar[dr] & & {B} \ar[dr] & \\
  & {K} \ar[rr] & &  {Y.}      }$$
Observe that the left side of the cube is a homotopy commutative triangle
$$\xymatrix{
{P_{\varphi ,p}} \ar[rr]^{\theta } \ar[dr]_{\overline{p}} & & {P_{\psi, q}} \ar[dl]^{\overline{q}} \\
 & {K.} &}$$
 Therefore, $\secat _{\psi }(q)=\secat (\overline{q})\leq \secat (\overline{p})=\secat _{\varphi }(p).$
\end{proof}

As a consequence of the previous result we have the following corollary:

\begin{corollary}\label{cor-fib}
Given any homotopy commutative diagram
$$\xymatrix{
 & {A} \ar[r]^{\gamma } \ar[d]_p &  {B} \ar[dl]^q \\
 {K} \ar[r]_{\varphi } & {X} &
}$$ \noindent we have that $\secat _{\varphi }(q)\leq \secat _{\varphi }(p).$
In other words, given maps $$K\stackrel{\varphi  }{\longrightarrow }X\stackrel{p}{\longleftarrow }A\stackrel{q}{\longleftarrow }B$$ \noindent we have
$\secat _{\varphi }(p)\leq \secat _{\varphi }(pq).$

\begin{proof}
Just apply the result above considering the particular case $\beta =1_X.$
\end{proof}
\end{corollary}

\begin{proposition}\label{varphi-psi}
Consider a homotopy commutative diagram of the form
$$\xymatrix{
 & {A} \ar[d]^p \\
 {K} \ar[r]^{\varphi } \ar[d]_{\lambda } & {X} \\
 {L.} \ar[ur]_{\psi } &
}$$ Then, $\secat _{\varphi }(p)\leq \secat _{\psi }(p).$
\end{proposition}

\begin{proof}
Consider the following notation for the corresponding homotopy pullbacks
$$\xymatrix{
{P_{\varphi ,p}} \ar[r]^{\overline{\varphi }} \ar[d]_{\overline{p}_0} & {A} \ar[d]^p &  & {P_{\psi ,p}} \ar[r]^{\overline{\psi }} \ar[d]_{\overline{p}_1} & {A} \ar[d]^p \\
{K} \ar[r]_{\varphi } & {X} & &  {L}\ar[r]_{\psi } & X.
}$$ From the weak universal property of the homotopy pullback consider the following induced whisker map
$$
\xymatrix{
{P_{\varphi ,p}}  \ar@/^1pc/[drr]^{\overline{\varphi }} \ar@/_1pc/[ddr]_{\lambda \overline{p}_0} \ar@{.>}[dr] & & \\
 & {P_{\psi ,p}} \ar[r]^{\overline{\psi}} \ar[d]_{\overline{p}_1} & {A} \ar[d]^p \\
 & {K} \ar[r]_{\psi } & {X}
}$$
\noindent so we obtain a homotopy commutative diagram
$$\xymatrix{
{P_{\varphi ,p}} \ar[r] \ar[d]_{\overline{p}_0} & {P_{\psi ,p}} \ar[d]^{\overline{p}_1} \ar[r]^{\overline{\psi }} & {A} \ar[d]^p \\
{X} \ar[r]_{\lambda } & {K} \ar[r]_{\psi } & {X.}
}$$ Since the right hand square is a homotopy pullback and the composition square is also a homotopy pullback, it follows that the left hand square is a homotopy pullback. Therefore, $$\secat _{\varphi }(p)=\secat (\overline{p}_0)\leq \secat (\overline{p}_1)=\secat _{\psi }(p).$$
\end{proof}

\begin{corollary}
Let $L\stackrel{\psi }{\longrightarrow }K\stackrel{\varphi }{\longrightarrow }X\stackrel{p}{\longleftarrow }A$  maps. Then,
$$\secat _{\varphi \psi }(p)\leq \secat _{\varphi }(p).$$
\end{corollary}

\bigskip
More interesting properties, this time related to $\secat(p)$ and $\cat(\varphi)$, are provided in the next result:

\begin{proposition}
Let $K\stackrel{\varphi }{\longrightarrow }X\stackrel{p}{\longleftarrow }A$ be a cospan of maps. Then:
\begin{enumerate}
\item $\secat _{\varphi }(p)\leq \secat (p)$ and $\secat _{\varphi }(p)=\secat (p)$
if $\varphi $ has a homotopy section.

\item If $X$ is path-connected, then $\secat _{\varphi }(p)\leq \cat (\varphi )$; if, in addition, $A$ is contractible, then $\secat _{\varphi }(p)=\cat (\varphi )$.
\end{enumerate}
\end{proposition}

\begin{proof}
From the homotopy pullback
$$\xymatrix{
{P_{\varphi ,p}} \ar[r]^{\overline{\varphi }} \ar[d]_{\overline{p}}  & {A} \ar[d]^p \\
{K} \ar[r]_{\varphi }  & {X}  }$$
\noindent we clearly observe that $\secat_{\varphi }(p)=\secat (\overline{p})\leq \secat (p)$.
If $s:X\rightarrow K$ is a homotopy section of $\varphi $, then, consider the following homotopy commutative diagram
$$\xymatrix{
 & {A} \ar[d]^p \\
 {X} \ar[r]^{id_X } \ar[d]_{s} & {X} \\
 {K.} \ar[ur]_{\varphi } &
}$$
Applying Proposition \ref{varphi-psi}, we obtain $\secat (p)\leq \secat _{\varphi }(p)$.
This substantiates assertion 1.

For the proof of assertion 2, as $X$ is path-connected we have a homotopy commutative diagram
$$\xymatrix{
 & {*} \ar[r] \ar[d] &  {A} \ar[dl]^p \\
 {K} \ar[r]_{\varphi } & {X} & .
}$$
Therefore, by Corollary \ref{cor-fib}, $\secat _{\varphi }(p)\leq \secat _{\varphi }(*\rightarrow X)=\cat (\varphi ).$ The second part comes from the fact that $*\rightarrow A$ is a homotopy equivalence and we can use Proposition \ref{h-invariant} in this diagram.
\end{proof}

\bigskip
Initially, if $\alpha, \beta,$ and $\gamma $ are not homotopy equivalences in Proposition \ref{h-invariant} above, drawing any conclusion is not possible. Nevertheless, given specific conditions, we can deduce a particular inequality:

\begin{proposition}\label{pav}
Consider the following homotopy commutative diagram
$$
\xymatrix{
{K} \ar[r]^{\varphi } \ar[d]_{\alpha } & {X} \ar[d]_{\beta}  & {A} \ar[d]_{\gamma } \ar[l]_{p} \\
{L} \ar[r]_{\psi }  & {Y}  & {B}  \ar[l]^{q} }
$$ \noindent where $K$ and $L$ have the homotopy type of a CW-complex. Let us suppose, in addition, that the subsequent conditions hold:
\begin{enumerate}
\item [(i)] $q$ is an $r$-equivalence ($r\geq 0$);

\item[(ii)] There exists $s\geq r$ such that $\alpha $ and $\beta $ are $(s+1)$-equivalences and $\gamma $ is an $s$-equivalence;

\item[(iii)] $(r+1)\secat _{\psi }(q)>\mbox{hdim}(K)-s+r$.
\end{enumerate}
Then, $\secat _{\varphi }(p)\leq \secat _{\psi }(q)$.

\end{proposition}

\begin{proof}
We have an induced homotopy commutative diagram
$$\xymatrix{
{P_{\varphi ,p}} \ar[r]^{\theta } \ar[d]_{\overline{p}} & {P_{\psi ,q }} \ar[d]^{\overline{q}} \\
{K} \ar[r]_{\alpha } & {L}
}$$ \noindent where, according to condition (ii), $\theta $ is an $s$-equivalence. Consequently, as $\alpha $ is an $(s+1)$-equivalence, the induced map $\widetilde{\theta }$ between the corresponding homotopy fibers is also an $s$-equivalence. This fact, together with conditions (i) and (iii), satisfies the criteria outlined by P. Pa\v{v}esi\'{c} in \cite[Cor. 2.9]{Pav}, resulting in the desired outcome.

\end{proof}

\begin{corollary}
Consider a homotopy commutative diagram of the form
$$\xymatrix{
 & {X} \ar[d]^{\beta } & {A} \ar[d]^{\gamma } \ar[l]_p \\
{K} \ar[ur]^{\varphi } \ar[r]_{\psi } & {Y} & {B.} \ar[l]^q
}$$
Let us suppose, in addition, that $K$ has the homotopy type of a CW-complex and the following conditions hold:
\begin{enumerate}
\item [(i)] $q$ is an $r$-equivalence ($r\geq 0$);

\item[(ii)] There exists $s\geq r$ such that $\gamma $ is an $s$-equivalence, and $\beta $ is an $(s+1)$-equivalence; and

\item[(iii)] $(r+1)\secat _{\psi }(q)>\mbox{hdim}(K)-s+r$.
\end{enumerate}
Then, $\secat _{\varphi }(p)=\secat _{\psi }(q)$.
\end{corollary}

\begin{proof}
Just consider Proposition \ref{ineq} and
Proposition \ref{pav} in the particular case $\alpha =id_K.$
\end{proof}

\bigskip
To conclude this section, we investigate a relationship concerning products. Let us suppose we have $K\stackrel{\varphi }{\longrightarrow }X\stackrel{p}{\longleftarrow }A$ and $L\stackrel{\psi }{\longrightarrow }Y\stackrel{q}{\longleftarrow }B$ as sets of continuous maps. In this scenario, we can construct the product
$$\xymatrix{
{K\times L} \ar[r]^{\varphi \times \psi } & {X\times Y} & {A\times B} \ar[l]_{p\times q}
}.$$

\begin{proposition}
Consider $K\stackrel{\varphi }{\longrightarrow }X\stackrel{p}{\longleftarrow }A$ and $L\stackrel{\psi }{\longrightarrow }Y\stackrel{q}{\longleftarrow }B$ maps where $K$ and $L$ are normal spaces. Then,
$$\secat _{\varphi \times \psi }(p\times q)\leq \secat _{\varphi }(p)+\secat _{\psi }(q).$$
\end{proposition}

\begin{proof}
Just take into account that we have a homotopy pullback
$$\xymatrix{
{P_{\varphi ,p}\times P_{\psi ,q}} \ar[rr] \ar[d]_{\overline{p}\times \overline{q}} & & {A\times B} \ar[d]^{p\times q} \\
{K\times L} \ar[rr]_{\varphi \times \psi } & & {X\times Y}
}$$ \noindent taken from the homotopy pullbacks $P_{\varphi ,p}$ and $P_{\psi ,q}$. Then, the result follows from the usual property of sectional category applied to a product of maps.
\end{proof}

\section{Whitehead-Ganea characterization of relative sectional category}

In their work, J. Gonz\'alez, M. Grant, and L. Vandembroucq \cite{GGV} had already considered the
characterization of Ganea type for relative sectional category. The purpose of this section is
to complement this study by incorporating the characterization of Whitehead type.
Recall that given any pair of maps $A\stackrel{f}{\longrightarrow}C\stackrel{g}{\longleftarrow }B$,
the \textit{join} of $f$ \textit{and} $g$, denoted as $A*_C B$, is the homotopy pushout of the
homotopy pullback of $f$ and $g$:
$$\xymatrix@C=0.7cm@R=0.7cm{ {\bullet } \ar[rr] \ar[dd] & & {B} \ar[dl] \ar[dd]^g \\
 & {A*_C B} \ar@{.>}[dr] & \\ {A} \ar[ur] \ar[rr]_f & & {C.} }$$
Here, the dotted arrow represents the corresponding co-whisker map induced by the weak universal
property of homotopy pushouts.
It is important to note that the map $A*_C B\to C$ is only defined up to weak equivalence. Any map constructed
in such a manner is weakly equivalent to the canonical co-whisker map obtained by first considering the standard
homotopy pullback of $f$ and $g$, followed by the standard homotopy pushout of the projections on $A$ and $B$.
To provide a concrete example, suppose that $f$ (or $g$) is a fibration. In such a case,
the honest pullback of $f$ and $g$ becomes a homotopy pullback. By taking the standard homotopy pushout of the
projections $A\times_C B\to A$ and $A\times_C B\to B$, we obtain a representative of the map $A*_C B\to C$ of the form:
$$\begin{array}{rcl}
A\ast_C B=A\amalg B \amalg (A\times_C B\times [0,1]) /\sim &\longrightarrow & C\\
&&\\
\langle a,b,t\rangle & \mapsto & f(a)=g(b)\\
\end{array}$$
where $\sim$ is defined by $(a,b,t)\sim a$ if $t=0$ and $(a,b,t)\sim b$ if $t=1$. It is worth noting that this explicit
construction coincides, up to homeomorphism, with the notion of sum of two fibrations used by Schwarz \cite{Sch}.

A Ganea-Schwarz-type characterization of the sectional category exists, outlined as follows:
Given any map $p:E\rightarrow B$, we can generate the join map $E*_B E\rightarrow B$ by combining $p$ with itself.
This process can be repeated iteratively, creating the join of this resulting map with $p$ once again. Thus, inductively,
we obtain $j^n_p:*^n_B E\rightarrow B$ as the join of $j^{n-1}_p$ and $p$ (with $j^0_p=p$ and $*^0_B E=E$).

\begin{theorem}\cite[Th. 2.2]{FGKV} \label{ax}
Let $p:E\rightarrow B$ be a map. If $B$ is normal,
then one has \secat$(p)\leq n$ if and only if $j^n_p$ admits a homotopy
section.
\end{theorem}

\begin{remark}
Using homotopy invariance, one can easily verify that if $p:E\rightarrow B$ is a map and $B$ has the homotopy type of a normal space (for instance, a space with the homotopy type of a CW-complex), then $\secat(p)\leq n$ if and only if $j^n_p$ admits a homotopy section.
\end{remark}

There is also a Whitehead-type characterization for the sectional category of a continuous map. This was provided by A. Fass\`{o}
\cite{F} and by J. Calcines and L. Vandembroucq in \cite{C-V}. The \emph{$n$-sectional fat wedge} of any map $p:E\rightarrow
B$
$$\kappa _n:T^n(p)\rightarrow B^{n+1}$$ \noindent is defined inductively as follows: setting $\kappa _0=p:E\rightarrow B$, and having $\kappa _{n-1}:T^{n-1}(p)\rightarrow B^n$ defined, we obtain $\kappa_n$ by considering the join of $\kappa _{n-1}\times id_B$ and $id_{B^n}\times p$:
$$\xymatrix@C=0.5cm@R=0.6cm{ {\bullet } \ar[rr] \ar[dd] & & {B^n\times E}
\ar[dl] \ar[dd]^{id_{B^n}\times p} \\
 & {T^n(p)} \ar@{.>}[dr]^{\kappa _n} & \\
 {T^{n-1}(p)\times B} \ar[ur] \ar[rr]_{\kappa _{n-1}\times id_B} & & {B^{n+1}.} }$$

\begin{theorem}\cite[Th. 8, Cor. 9]{C-V}\label{conex}
Let $p:E\rightarrow B$ be any map and $n\geq 0$ be any nonnegative integer. Then, there exists a homotopy pullback
$$\xymatrix{
{*^n_B E} \ar[d]_{j^n_p} \ar[rr] & &
{T^n(p)} \ar[d]^{\kappa _n} \\ {B} \ar[rr]_{\Delta _{n+1}} & &
{B^{n+1}.} }$$
As a consequence, if $B$ is a normal space, then $\secat (p)\leq n$ if, and
only if, there is, up to homotopy, a lift of the $(n+1)$-diagonal
map
$$\xymatrix{
{} & {T^n(p)} \ar[d]^{\kappa _n} \\
{B} \ar@{.>}[ur] \ar[r]_{\Delta _{n+1}} & {B^{n+1}.} }$$
\end{theorem}

Applying this Whitehead-Ganea-type characterization to the sectional category, we obtain a characterization of the same type for the relative sectional category.

\begin{proposition}
Let $A\stackrel{p}{\longrightarrow }X\stackrel{\varphi }{\longleftarrow }K$ be a cospan of continuous maps, where $K$ is a normal space. Then, the following statements are equivalent:
\begin{enumerate}
\item $\secat _{\varphi }(p)\leq n$.

\item There exists a map $\sigma :K\rightarrow *^n_XA$ such that $j^n_p\circ \sigma \simeq \varphi $:
$$\xymatrix{
 & {*^n_XA} \ar[d]^{j^n_p} \\
 {K} \ar@{.>}[ur]^{\sigma } \ar[r]_{\varphi } & {X.}
}$$

\item There exists a map $\sigma :K\rightarrow T^n(p)$ such that $\kappa _n\circ \sigma \simeq \Delta _{n+1}\circ \varphi $:
$$\xymatrix{
{K} \ar@{.>}[r]^(.4){\sigma } \ar[d]_{\varphi } &  {T^n(p)} \ar[d]^{\kappa _n} \\
 {X} \ar[r]_{\Delta _{n+1}} & {X^{n+1}.}
}$$
\end{enumerate}
\end{proposition}

\begin{proof}
The equivalence between 1. and 2. comes from Theorem \ref{ax} and the weak universal property of the following homotopy pullback:
$$\xymatrix{
{*^n_K P_{\varphi ,p}} \ar[rr] \ar[d]_{j^n_{\overline{p}}} & & {*^n_X A} \ar[d]^{j^n_p} \\
{K} \ar[rr]_{\varphi } & & {X.} }$$
The equivalence between 2. and 3. is due the weak universal property of the homotopy pullback established in Theorem \ref{conex} applied to the map $p:A\rightarrow X$.
\end{proof}

\section{Generalized relative sectional category}
The study of the relative sectional category becomes more intriguing when we broaden our scope beyond open sets to include arbitrary subsets in its definition.
To this end, we define $\secat_ g(p)$, the generalized sectional category of a continuous map $p:E\rightarrow B$, as the smallest non-negative integer $n$ (or infinity if such an integer does not exist) such that $B$ can be covered by $n+1$ subsets, on each of which there exists a local homotopy section of $p.$ This is a homotopy invariant, as the next result assures:

\begin{proposition}\label{invariance}
Let us consider the following homotopy commutative diagram, in which $\alpha$ and $\beta$ are homotopy equivalences:
$$\xymatrix{
{X} \ar[r]^{\alpha }_{\simeq } \ar[d]_f & {X'} \ar[d]^{f'} \\
{Y} \ar[r]^{\simeq }_{\beta } & {Y'.} }$$ In this situation, we have $\secat_g(f) = \secat_g(f')$ (and $\secat (f) = \secat (f')$).
\end{proposition}

In many cases, generalized sectional category aligns with the conventional sectional category, particularly under conditions that are not overly restrictive.

\begin{theorem}\cite[Th. 2.7]{GC}\label{chulo}
Let $p:E\rightarrow B$ be a fibration where $E$ and $B$ are ANR spaces. Then, $\secat(p)=\secat _g(p).$
\end{theorem}

\begin{remark}\label{Miyata}
The preceding outcome applies equally to any continuous map $f:X\rightarrow Y$ between ANR spaces, not necessarily a fibration. A result by T. Miyata \cite[Th. 2.1]{Miy} asserts that $f$ can be decomposed as $f=pq$, where $q:X\rightarrow E$ is a homotopy equivalence with $E$ being an ANR space, and $p:E\rightarrow Y$ is a map with a property slightly stronger than the typical homotopy lifting property required for a Hurewicz fibration. Specifically, $p$ exhibits the so-called ``strong homotopy lifting property" (SHLP) with respect to every space. This implies that for any commutative diagram with solid arrows as described,
$$\xymatrix{
{Z} \ar[r]^h \ar[d]_{i_0} & {E} \ar[d]^p \\
{Z\times I} \ar[r]_H \ar@{.>}[ur]^{\widetilde{H}} & {Y}
}$$ \noindent there exists a map $\widetilde{H}:Z\times I\rightarrow E$ such that $\widetilde{H}i_0=h,$ $p\widetilde{H}=H$ and $\widetilde{H}$ is constant on $\{z\}\times I$ whenever $H$ is constant on $\{z\}\times I$.

Utilizing the homotopy invariance of $\secat $ and its generalized counterpart, as established in Proposition \ref{invariance} and Theorem \ref{chulo}, we deduce $\secat (f)=\secat _g(f)$. Extending this argument further and once again leveraging the homotopy invariance of $\secat$ and $\secat_g$, we can conclude that Theorem \ref{chulo} remains valid even when $f$ is any continuous map between spaces with the homotopy type of a CW-complex (or an ANR space).

All these comments are summarized in the next proposition.
\end{remark}

\begin{proposition}\label{fantastico}
Let $f:X\rightarrow Y$ be a map between spaces having the homotopy type of a CW-complex (or an ANR space). Then, $\secat(f)=\secat _g(f).$ \hfill $\square$
\end{proposition}

\bigskip
Now we want to extend the results above to the relative case:

\begin{definition}
Consider a cospan $K\stackrel{\varphi }{\longrightarrow }X\stackrel{p}{\longleftarrow }A$ of maps.
The \emph{generalized sectional category of $p$ relative to $\varphi $}, $\secat _{g;\varphi }(p)$, is the least non-negative integer $n$ such that there is a cover $K=Z_0\cup Z_1\cup \cdots \cup Z_n$ by $n+1$ $\varphi $-sectional subsets. Is such an integer does not exist then we set $\secat  _{g;\varphi }(p)=\infty .$
\end{definition}

A similar proof to that given in Proposition \ref{pullback} gives us the following result:

\begin{proposition}
Let $K\stackrel{\varphi }{\longrightarrow }X\stackrel{p}{\longleftarrow }A$ be a cospan of maps. Then $\secat _{g;\varphi }(p)=\secat _g(\overline{p})$, where $\overline{p}$ denotes the base change of the map $p$ in the homotopy pullback of $p$ along $f$:
$$\xymatrix{
{P_{\varphi ,p}} \ar[r]^{\overline{\varphi }} \ar[d]_{\overline{p}}  & {A} \ar[d]^p \\
{K} \ar[r]_{\varphi }  & {X.}  }$$ \hfill $\square$
\end{proposition}

In particular, following a similar proof as the one given in in Proposition \ref{h-invariant}, one can easily check that $\secat _{g;\varphi }(p)$ is a homotopy invariant:

\begin{proposition}
Consider the following homotopy commutative diagram
$$
\xymatrix{
{K} \ar[r]^{\varphi } \ar[d]^{\simeq }_{\alpha } & {X} \ar[d]^{\simeq }_{\beta}  & {A} \ar[d]^{\simeq }_{\gamma } \ar[l]_{p} \\
{L} \ar[r]_{\psi }  & {Y}  & {L}  \ar[l]^{q} }
$$
where $\alpha ,\beta $ and $\gamma $ are homotopy equivalences. Then $\secat _{g;\varphi }(p)=\secat _{g;\psi}(q).$

\hfill $\square$
\end{proposition}

The subsequent result, which is significant for our objectives, is widely recognized. Nonetheless, due to the absence of an explicit proof in the existing literature (to the best of my knowledge), I have found it pertinent to provide a demonstration. We first introduce a lemma:

\begin{lemma}\label{lema1}
Let $X$ be a metrizable space and $f,g:X\rightarrow Z$ maps where $Z$ is an ANR space. Suppose $A\subset X$ is a closed subspace and there exists a homotopy $H:f_{|A}\simeq g_{|A}.$ Then, there exists an open subset $U\supset A$ and a homotopy $\widetilde{H}:f_{|U}\simeq g_{|U}$ such that $\widetilde{H}_{|A\times I}=H.$
\end{lemma}

\begin{proof}
We define $\Phi :X\times \{0,1\}\cup A\times I\rightarrow Z$ as follows: $\Phi (x,0)=f(x),$ $\Phi (x,1)=g(x)$ and $\Phi (a,t)=H(a,t)$, for all $x\in X,$ $a\in A,$ $t\in I.$ As $Z$ is an ANR space, we can consider an open subset $W\supset X\times \{0,1\}\cup A\times I$ in $X\times I$ and an extension $\widetilde{\Phi }:W\rightarrow Z$
$$\xymatrix{
{X\times \{0,1\}\cup A\times I} \ar[rr]^(.6){\Phi } \ar@{^{(}->}[d] & & {Z} \\
{W} \ar[urr]_{\widetilde{\Phi}} & & .
}$$
Utilizing the compactness of $I$ we find an open subset $U\subset X$ such that $U\times I\subset W.$ Then, the restriction $\widetilde{H}:=\widetilde{\Phi }_{|U\times I}:U\times I\rightarrow Z$ satisfies the required conditions.
\end{proof}

\begin{proposition}\label{ANR}
Let us consider a cospan of maps $X\stackrel{f}{\longrightarrow}Z\stackrel{g}{\longleftarrow}Y$, where $X$, $Y$ and $Z$ are ANR spaces. Then,
$$E_{f,g}:=\{(x,y,\alpha)\in X\times Y\times Z^I:\alpha(0)=f(x), \alpha(1)=g(y)\}$$
\noindent the standard homotopy pullback of $f$ and $g$, is also an ANR.
\end{proposition}

\begin{proof}
We observe that $E_{f,g}$, as a subspace of $X\times Y\times Z^I$, is metrizable.
Now, for a metrizable space $M$ and a map $\phi :A\rightarrow E_{f,g}$ from a closed subspace $A\subset M$, let $\phi _X:A\rightarrow X$, $\phi_Y:A\rightarrow Y$ and $\phi_Z:A\rightarrow Z^I$ be the maps derived from $\phi $ by composition with the corresponding projections.
Since $X$ and $Y$ are ANR spaces, the maps $\phi _X$ and $\phi _Y$ possess respective extensions
$\widetilde{\phi }_X:U_X\rightarrow X$ and $\widetilde{\phi}_Y:U_Y\rightarrow Y,$ where $U_X,U_Y\subset M$ are open neighbourhoods of $A$.

Considering $H:A\times I\rightarrow Z$ as the adjoint map of $\phi _Z$, we clearly have a homotopy $H:f\phi _X\simeq g\phi _Y$. Consequently
$$f(\widetilde{\phi}_X)_{|A}=f\phi_X\simeq g\phi_Y=g(\widetilde{\phi}_Y)_{|A}.$$
Let $N:=U_X\cap U_Y$, and without loss of generality, we can consider $f\widetilde\phi_X,g\widetilde\phi_Y :N\rightarrow Z$.
Notice that $A$ is a closed subspace of $N$. Then, using Lemma \ref{lema1} above, we find an open neighbourhood $U\supset A$ and a homotopy
$\widetilde{H}:U\times I\rightarrow Z$ such that $\widetilde{H}:f(\widetilde{\phi}_X)_{|U}\simeq g(\widetilde{\phi}_Y)_{|U}$ and $\widetilde{H}_{|A\times I}=H.$
Additionally, since $N$ is open in $M$, $U$ is also an open neighbourhood of $A$ in $M$.

Finally, by defining $\widetilde{\phi }_Z:U\rightarrow Z^I$ as the adjoint map of $\widetilde{H}$ it is straightforward to verify that $\widetilde{\phi }:U\rightarrow E_{f,g}$, defined as $\widetilde{\phi }:=(\widetilde{\phi }_X,\widetilde{\phi }_Y,\widetilde{\phi }_Z)$, establishes an extension
$$\xymatrix{
{A} \ar[r]^{\phi } \ar@{^{(}->}[d] & {E_{f,g}} \\
{U} \ar@{.>}[ur]_{\widetilde{\phi }} & .
}$$
\end{proof}

\begin{proposition}
If $K\stackrel{\varphi }{\longrightarrow }X\stackrel{p}{\longleftarrow }A$ is a cospan of maps, where $K$, $X$ and $A$ have the homotopy type of an ANR space (or, equivalently, of a CW-complex), then,
$$\secat _{\varphi }(p)=\secat _{g; \varphi }(p).$$
\end{proposition}

\begin{proof}
One can construct a homotopy commutative diagram
$$
\xymatrix{
{K} \ar[r]^{\varphi } \ar[d]^{\simeq } & {X} \ar[d]^{\simeq }  & {A} \ar[d]^{\simeq } \ar[l]_{p} \\
{L} \ar[r]_{\psi }  & {Y}  & {B}  \ar[l]^{q} }
$$ \noindent where the vertical maps are homotopy equivalences and $L$, $Y$ and $B$ are ANR spaces. This diagram gives rise to another homotopy commutative diagram
$$\xymatrix{
{P_{\varphi ,p}} \ar[r]^{\simeq } \ar[d]_{\overline{p}} & {E_{\psi, q}} \ar[d]^{\overline{q}}  \\
{K} \ar[r]_{\simeq } & {L.} }$$ \noindent By Proposition \ref{ANR} we conclude that $P_{\varphi ,p}$ also has the homotopy type of an ANR space. Therefore, by Remark \ref{Miyata} and/or Proposition \ref{fantastico}, we complete the proof.
\end{proof}

\bigskip
\textbf{Acknowledgments.}
I would like to express my gratitude to Professor T. Cutler for generously providing me with the result that the standard homotopy pullback of continuous maps between ANR spaces is also an ANR space. I also gratefully acknowledge Professor A. Murillo for his careful reading and valuable comments on a preliminary version of this work.

\bigskip
\textbf{Funding.}
This work is supported by the project PID2020-118753GB-I00 from the Spanish Ministry of Science and Innovation.

\end{document}